\documentclass[10pt,a4paper,twoside]{article}
\usepackage{amsmath,amssymb,enumerate,%bbm
,theorem}
\usepackage{epsfig,color,graphics,graphicx}
\usepackage{amssymb,amsbsy,amsmath,amsfonts,amssymb,amscd}
\usepackage[english]{babel}
\newcommand{\assign}{:=}
\newcommand{\tmdummy}{\ensuremath{\text{}}}
\newcommand{\tmem}[1]{{\em #1\/}}
\newcommand{\tmop}[1]{\ensuremath{\operatorname{#1}}}
\newcommand{\tmtextbf}[1]{{\bfseries{#1}}}
\newcommand{\tmtextit}[1]{{\itshape{#1}}}
\newcommand{\tmtextsc}[1]{{\scshape{#1}}}

\newenvironment{itemizedot}{\begin{itemize}  }{\end{itemize}}
\newenvironment{proof}{\noindent\textbf{Proof\ }}{\hspace*{\fill}$\Box$\medskip}
\newtheorem{theorem}{Theorem}

\newtheorem{definition}[theorem]{Definition}

\newtheorem{lemma}[theorem]{Lemma}
\newtheorem{proposition}[theorem]{Proposition}
\newtheorem{remark}{Remark}
\begin{document}

\title{Controllability and observability of an artificial advection-diffusion
problem }
\author{Pierre Cornilleau\thanks{Teacher at Lyc\'ee du parc des Loges, 1, boulevard des Champs-\'Elys\'ees, 91012 \'Evry, France.\newline e-mail: pierre.cornilleau@ens-lyon.org.}  \& Sergio Guerrero \thanks{Laboratoire Jacques-Louis Lions,
Universit\'e Pierre et Marie Curie, 75252 Paris C\'edex 05, France. \newline e-mail : guerrero@ann.jussieu.fr.}}\maketitle

\begin{abstract}
In this paper we study the controllability of an artificial advection-diffusion system through the boundary.
Suitable  Carleman estimates give us the  observability of the adjoint system in the one dimensional case.
We also study some basic properties of our problem such as backward uniqueness and we get an intuitive result on the control cost for vanishing viscosity.
\end{abstract}

\tableofcontents

{\section*{Introduction}}

\paragraph{Artificial advection-diffusion problem}

In the present paper we deal with an advection-diffusion problem with small viscosity
truncated in one space direction.
 Our interest for the linear advection diffusion equation comes from the Navier-Stokes equations, but it arises also from other fields as,
 for example, meteorology.
For a given viscosity $\varepsilon > 0$, the incompressible Navier-Stokes equations can be written as
\[ \left\{ \begin{array}{c}
     f_t + (f.\nabla) f - \varepsilon \Delta f +\nabla p= 0\\
     \text{div}(f) = 0\\
   \end{array} \right. \]
where $f$ is the velocity vector field, $p$ the pressure, $\nabla$ the gradient and $\Delta$ the usual Laplacian.
Considering the flow around a body, we have that $f$ is almost constant far away from the body and equal to $a$. Then, our system can be approximated by
the Oseen equation
\[ \left\{ \begin{array}{c}
     f_t + (a.\nabla) f - \varepsilon \Delta f +\nabla p= 0\\
     \text{div}(f) = 0\\
   \end{array} \right. \]
(see, for instance, \cite[page 309, (1.2)]{HS}). Consequently, we get the following equation for the vorticity $u=\text{rot}(f)$
\begin{equation*}
u_t+a.\nabla u - \varepsilon \Delta u =0.
\end{equation*}
In the sequel, we assume for simplicity that $a$ is the $n$th unit vector of the canonical basis of  $\mathbb{R}^{n}$.
\par
When one computes the solution of this problem, it can only be solved numerically on a bounded domain. A good way to approximate the solution on
the whole space may be given by the use of artificial boundary conditions (see
{\cite{H,HS}}).

In this paper, we will consider an advection-diffusion system in a strip $\Omega:=\mathbb{R}^{n-1} \times (-L,0)$ ($L$ some positive constant) with particular artificial boundary conditions on both sides of the domain. The following system was considered in \cite[section 6]{H}:
\begin{equation}\label{systemeHalpern} \left\{ \begin{array}{c}
     u_t + \partial_{x_n} u - \varepsilon \Delta u = 0\\
     \varepsilon(u_t + \partial_{\nu} u) = 0\\
     \varepsilon (u_t + \partial_{\nu} u) + u = 0\\
     u (0, .) = u_0
   \end{array} \right. \begin{array}{c}
     \text{in } (0, T) \times \Omega,\\
     \text{on } (0, T) \times \Gamma_0,\\
     \text{on } (0, T) \times \Gamma_1,\\
     \text{in } \Omega,
   \end{array} \end{equation}
where $T>0$, $\Gamma_0 \assign \mathbb{R}^{n - 1} \times \{0\}$, $\Gamma_1
\assign \mathbb{R}^{n - 1} \times \{- L\}$ and we have denoted $\partial_{x_n}$ the partial derivative with respect to $x_n$ and $\partial_{\nu}$ the normal derivative. In \cite[Theorem 3]{H}, the author proved that the solution of \eqref{systemeHalpern} converges in some sense to the restriction of the solution of
$$
\left\{ \begin{array}{c}
     u_t + \partial_{x_n} u  = 0\\
     u = 0\\
     u (0, .) = u_0
   \end{array} \right. \begin{array}{c}
     \text{in } (0, T) \times \mathbb{R}^{n-1}\times \mathbb{R}_-,\\
     \text{on } (0, T) \times \Gamma_1,\\
     \text{in } \mathbb{R}^{n-1}\times \mathbb{R}_-,
   \end{array}
   $$
   to $(0,T)\times\Omega$ (see also section 6 in \cite{H}).

    Up to our knowledge, there is no controllability result concerning system (\ref{systemeHalpern}). In this paper, we will be interested in the uniform boundary controllability of (\ref{systemeHalpern}):
    $$
    \left\{ \begin{array}{c}
     u_t + \partial_{x_n} u - \varepsilon \Delta u = 0\\

     \varepsilon (u_t + \partial_{\nu} u) + u1_{\Gamma_1} = v1_{\Gamma_i}\\
     u (0, .) = u_0
   \end{array} \right. \begin{array}{c}
     \text{in } (0, T) \times \Omega,\\
     \text{on } (0, T) \times \partial\Omega,\\
     \text{in } \Omega,
   \end{array}
   $$
   where $i=0$ or $i=1$.

In the sequel, we shall focus on the one-dimensional problem
\[ (S_v) \left\{ \begin{array}{c}
     u_t + u_x - \varepsilon u_{x x} = 0\\
     \varepsilon(u_t + \partial_{\nu} u) = v\\
     \varepsilon (u_t + \partial_{\nu} u) + u = 0\\
     u (0, .) = u_0
   \end{array} \right. \begin{array}{c}
     \text{in } (0, T) \times (-L,0),\\
     \text{on } (0, T) \times \{0\},\\
     \text{on } (0, T) \times \{-L\},\\
     \text{in } (-L,0).
   \end{array} \]
 We are interested in the so-called \textit{null controllability} of this system
\[ \tmop{for} \tmop{given} u_0, \tmop{find} v \tmop{such} \tmop{that}
   \tmop{the} \tmop{solution} \tmop{of} (S_v) \tmop{satisfies} u (T) \equiv 0. \]
This property is classicaly equivalent to an observability inequality
for the so-called {\it adjoint system} (see Proposition \ref{ObsControl} below for a proof), which can be written as
\[ (S') \left\{ \begin{array}{c}
     \varphi_t + \varphi_x + \varepsilon \varphi_{ x x} = 0\\
     \varepsilon (\varphi_t - \partial_{\nu} \varphi) - \varphi = 0\\
     \varphi_t - \partial_{\nu} \varphi = 0\\
     \varphi (T, .) = \varphi_T
   \end{array} \right. \begin{array}{c}
     \text{in } (0, T) \times (-L,0),\\
     \text{on } (0, T) \times \{ 0\},\\
     \text{on } (0, T) \times \{-L\},\\
     \text{in } (-L,0).
   \end{array} \]
The observability inequality corresponding to the previous controllability property is:
\begin{equation}\label{OI}
\text{ there exists $C>0$ such that } \Vert\varphi(0,.)\Vert_X \leq C \Vert\varphi(.,0)\Vert_{L^{2}(0,T)} \quad \forall \varphi_{T} \in X,
\end{equation}
where the space $X$ will be defined below. We refer the reader unfamiliar with the links between observability and controllability to the seminal paper \cite{DR} (see also \cite{JLL} on hyperbolic systems).
\par

The phenomenon of degeneration of a parabolic problem to a hyperbolic one has been studied in several papers: see, for instance, \cite{CG} or \cite{Glass} (one dimensional heat equation) and \cite{GG} (Burgers equation). Similar results of
interest can also be found in \cite{C}.

\paragraph{Main results}

We define $X$  as the closure of ${\cal D}(\bar{\Omega}):=\mathcal{C}^{\infty} (
\bar{\Omega}$) for the norm
$$\|u\|_X \assign \left( \|u\|^2_{L^2(\Omega)} +
\varepsilon \|u\|^2_{L^2(\partial\Omega)} \right)^{\frac{1}{2}}.$$
 Observe that $X$ is the largest subspace of $L^2(\Omega)$ such that its elements possess a trace in $L^2(\partial\Omega)$. In particular,
$$\forall \epsilon>0, \ H^{1/2+\epsilon}(\Omega)\varsubsetneq X \varsubsetneq L^2(\Omega).$$
\par We will denote by $C_{obs} (\varepsilon)$ the cost of the null-control, which is the
smallest constant $C$ which fulfills the observability estimate \eqref{OI}. Our main result is the following:
\begin{theorem}\label{Seville}
 Assume $T > 0$.
 \begin{itemize}
  \item For any $\varepsilon>0$, there exists $\overline C > 0$ such that the
  solution of problem $(S')$ satisfies (\ref{OI}). Consequently, for every $u_0 \in X$, there exists a control
  $v \in L^2 (0,T)$ with $$\|v\|_{L^2 (0,T)} \leq \overline C\|u_0 \|_X$$ such that the solution of the problem $(S_v)$
  satisfies $u (T) \equiv 0$.
  \item Furthermore, if $T / L$ is large enough, the cost of the null-control $C_{obs} (\varepsilon)$ tends to zero exponentially as
  $\varepsilon \rightarrow 0$:
$$\exists C,k>0 \text{ such that } \quad C_{obs}(\varepsilon)\le C e^{-k/\varepsilon} \quad \forall \varepsilon\in(0,1).$$
 \end{itemize}
\end{theorem}

\begin{remark}
One can in fact obtain an observability result for the adjoint system at $x=-L$, that is
\begin{equation*}
\| \varphi (0,.)\|_X \leq C\| \varphi(.,-L) \|_{L^2(0, T)},
\end{equation*}
for any $\varphi$ solution of $(S')$. This provides some controllability result for the direct system on $\Gamma_1$: a function $v$ can be found depending continuously on $u_0$ so that the solution of
\[ \left\{ \begin{array}{c}
     u_t + u_x - \varepsilon u_{x x} = 0\\
     u_t + \partial_{\nu} u =0 \\
     \varepsilon (u_t + \partial_{\nu} u) + u = v\\
     u (0, .) = u_0
   \end{array} \right. \begin{array}{c}
     \text{in } (0, T) \times (-L,0),\\
     \text{on } (0, T) \times \{0\},\\
     \text{on } (0, T) \times \{-L\},\\
     \text{in } (-L,0),
   \end{array}  \]
   satisfies $u(T,.) \equiv 0$. Despite the fact that it is more physical to control our system at $\Gamma_1$,  we have chosen to present here the result for the system $(S_v)$ since its proof, less obvious, requires more computations.
\end{remark}

\begin{remark}
  The fact that the control cost vanishes tells intuitively that the
  state is almost null for $T / L$ big enough. This is to be connected
  with the fact that, for $\varepsilon = 0$, the system is purely advective
  and then that, for $T > L$, its state is null.
\end{remark}

In some context of inverse problems (be able to know the origin of a polluted
river for instance), it can be interesting to know if the observation of the
solution of the direct problem $(S_0)^n$ on the boundary part $\Gamma_1$ or $\Gamma_0$ can allow us to
recover the initial data. The corresponding result is presented now and will be proved at the end of the first section.

\begin{proposition}\label{OBS}
  Let $T > 0$ and $\varepsilon>0$. If the solution of $(S_0)$ with initial data $u_0 \in X$
  satisfies $u = 0$ on $(0, T) \times \Gamma_{1}$, then $u_0 \equiv 0$. However, there is
  no constant $C > 0$ such that the following estimate holds
  \begin{equation}\label{FE}
   \|u_0 \|_X \leq C\|u(.,-L)\|_{L^2 (0,T)}\quad\forall u_0 \in X.
  \end{equation}
\end{proposition}

\begin{remark}
Using Remark 1, we can also obtain a similar result at $x=0$.
\end{remark}

The rest of the article is organized as follows: in the first section, we show
the well-posedness of  the direct and the adjoint problems using some semi-groups
approach. In the second section, we adapt a Carleman inequality to the case of our one-dimensional problem. The third section is intended to explain how to get observability in the one-dimensional case and the equivalence between observability and controllability.

\par $\ $

\par \textbf{Notations:}
\par \noindent $A \lesssim B$ means that, for some universal constant $c>0$, $A\leq c B.$
\par \noindent $A \sim B$ means that, for some universal constant $c>1$, $c^{-1} B \leq A \leq c B.$

\section{Well-posedness and basic properties of systems}
In this section, we work in dimension $n$.
\subsection{Homogeneous problems}

We will use some semi-group results to show existence and uniqueness of the
homogeneous direct problem (that is $(S_0)^n$). This will enable to define
solutions of the system $(S_0)^n$ as a semigroup value.
We define $ H = (X, \|. \|_X), V = H^1 (\Omega)$
endowed with the usual norm $\|.\Vert$ and we consider the
 bilinear form on $V$ defined by
\begin{equation}\label{a-def}
 a (u_1, u_2) = \varepsilon\int_{\Omega} \nabla u_1. \nabla u_2+\int_{\Omega} \partial_{x_n} u_1 u_2 +\int_{\Gamma_1} u_1 u_2 .
\end{equation}
With the help of this bilinear form, one may now consider the space
\begin{equation*}
{\cal D}:=\left\{u_1\in X; \sup_{u_2 \in {\cal D}(\overline{\Omega}); \ ||u_2||_X\le 1} |a(u_1,u_2)|<+\infty \right\}
\end{equation*}
 equipped with the natural norm
$$\| u_1\| _{\mathcal{D}}=\| u_1\|_X+\sup_{u_2 \in {\cal D}(\overline{\Omega}); \ ||u_2||_X\le 1} |a(u_1,u_2)|.$$
Note that, using an integration by parts, one shows that $a(u_1,u_2)$ is well-defined for $u_1\in X$ and $u_2\in {\cal C}^\infty(\overline{\Omega})$ and that the map
$$ u_2 \in X\mapsto a(u_1,u_2) \in \mathbb{R}$$
is well-defined and continuous for any $u_1 \in {\cal D}$.
\newline \noindent Using the Riesz representation theorem, we can define an operator ${\cal A}$ with domain ${\cal D(A)}={\cal D}$ and such that
$$\forall u_1 \in {\cal D(A)}, \ \forall u_2\in X, \quad <-{\cal A}u_1, u_2>_X=a(u_1,u_2).$$
$(S_0)^n$ might now be written in the following
abstract way
\[ \left\{ \begin{array}{c}
    u_t = \mathcal{A} u,\\
     u(0, .) = u_0,
   \end{array} \right.
% \begin{array}{c}
%      \text{in } X,\\
%      \text{in } X,\\
%   \end{array}
 \]
since, for $u\in {\cal C}(\mathbb{R}^+, {\cal D(A)}) \cap {\cal C}^1(\mathbb{R}^+, X)$ a classical solution of $(S_0)^n$, we directly get that
$$\forall v\in X, \quad <u_t,v>_X=-a(u,v).$$
\begin{proposition}
Let $\varepsilon>0$. Then $\mathcal{A}$ generates a continuous semi-group $(e^{t\mathcal{A}})_{t\geq 0}$ on $X$.
\end{proposition}
\begin{proof}
Using the Hille-Yoshida theorem, it is sufficient to show that $\mathcal{A}$ is a maximal monotone operator.
\begin{itemize}
  \item First, Green-Riemann formula gives
  \begin{equation*}
  <\mathcal{A}u,u>_X=-\varepsilon\int_{\Omega} |\nabla u|^2-\frac{1}{2}\int_{\partial\Omega}|u|^2,
  \end{equation*}
  so the monotonicity is proved.
  \item Given $v \in X$, we have to solve $(I-\mathcal{A})u=v$, that is to find $u\in \mathcal{D} (
\mathcal{A})$ such that
$$ \forall u'\in X, <(I-\mathcal{A})u,u'>_X=<v,u'>_X.$$
The left-hand side term of this equation is a continuous bilinear form $B$ on the space $V$ while the right-hand side term is a continuous linear form $L$ on $V$. Moreover, using \eqref{a-def}, we can easily compute
$$ B(u,u)=\Vert u\Vert_X^2+\varepsilon \int_\Omega \vert\nabla u \vert^2+\frac{1}{2}\int_{\partial\Omega} \vert u \vert^2\geq \min\{1,\varepsilon\} \Vert u\Vert^2,$$
which show that $B$ is coercive on $V$.
 Consequently, Lax-Milgram theorem shows that there exists $u\in V$ such that $B(u,u')=L(u'), \ \forall u'\in V $. Using test functions $u'\in \mathcal{D} (
\Omega)$ additionaly gives $u\in \mathcal{D} (\mathcal{A})$ and our proof ends.
\end{itemize}
\end{proof}

\noindent In particular, for every initial data $u_0 \in
X$, we have existence and uniqueness of a solution $u \in \mathcal{C} ( \mathbb{R}^+, X)$ to $(S_0)^n$. We will
call these solutions  \tmem{weak solutions} opposed to \tmem{strong
solutions} \ i.e. such that $u_0 \in \mathcal{D(A)}$ and which fulfill $
u \in \mathcal{C} ( \mathbb{R}^+, \mathcal{D(A)}) \cap \mathcal{C}^1 (\mathbb{R}^+, X)$. One can notice that, using the density of $\mathcal{D(A)}$ in $X$, a weak
solution can always be approximated by a strong solution.

Let us introduce the adjoint system in dimension $n$
\[ (S')^n \left\{ \begin{array}{c}
     \varphi_t + \partial_{x_n} \varphi + \varepsilon \Delta \varphi = 0\\
     \varepsilon (\varphi_t - \partial_{\nu} \varphi) - \varphi = 0\\
     \varphi_t - \partial_{\nu} \varphi = 0\\
     \varphi (T, .) = \varphi_T
   \end{array} \right. \begin{array}{c}
     \text{in } (0, T) \times \Omega,\\
     \text{on } (0, T) \times \Gamma_0,\\
     \text{on } (0, T) \times \Gamma_1,\\
     \text{in } \Omega.
   \end{array} \]
Defining the adjoint ${\cal A}^*$ of ${\cal A}$, we can show in a similar way as before that $(S')^n$ may be written in the following abstract way
\[ \left\{ \begin{array}{c}
    \varphi_t + \mathcal{A^*} \varphi=0,\\
     \varphi(T, .) = \varphi_T.
   \end{array} \right.
%  \begin{array}{c}
%      \text{in } X,\\
%      \text{in } X.\\
%    \end{array}
\]
The adjoint operator $\mathcal{A^{\ast}}$ is also a maximal monotone operator with domain $\mathcal{D(A^*)}$. Thus, for every initial data $\varphi_T \in
X$, we have existence and uniqueness  of a solution $\varphi \in \mathcal{C} ( \mathbb{R}^+, X)$ to $(S')^n$ given by means of the backward semigroup $(e^{(T-t)\mathcal{A^*}})_{t\geq 0}$. We will also speak of \tmem{weak solutions} or \tmem{strong
solutions} in this situation.

\subsection{Nonhomogeneous direct problems}

We consider a slightly more general system
\[(S_{f,g_0,g_1})^n \left\{ \begin{array}{c}
     u_t + \partial_{x_n} u - \varepsilon \Delta u = f\\
     \varepsilon(u_t + \partial_{\nu} u) = g_0\\
     \varepsilon (u_t + \partial_{\nu} u) + u = g_1\\
     u (0, .) = u_0
   \end{array} \right. \begin{array}{c}
     \text{in } (0, T) \times \Omega,\\
     \text{on } (0, T) \times \Gamma_0,\\
     \text{on } (0, T) \times \Gamma_1,\\
     \text{in } \Omega,
   \end{array}  \]
with $f\in L^2((0,T)\times\Omega)$, $g_0\in L^2((0,T)\times\Gamma_0)$ and $g_1\in L^2((0,T)\times\Gamma_1)$. We consider a function $g$ on $\partial\Omega$ such that $g=g_i$ on $\Gamma_i$ ($i=0,1$).
\begin{definition}\label{Ladefinition}
  We say that $u \in \mathcal{C} ( \left[ 0, T \right], X)$ is a {\it transposition} solution of
  $(S_{f,g_0,g_1})^n$ if, for every function $\varphi \in \mathcal{C} ([0, T], {\cal D(A^*)}) \cap \mathcal{C}^1 ([0, T ], X)$, the following
  identity holds
  \[ \int_0^{\tau} \left( <u, \varphi_t>_X+<u,{\cal A^*} \varphi>_X + <F,\varphi>_X\right) = \left[ <u(t),\varphi(t)>_X\right]_{t=0}^{t=\tau}
  \quad \forall \tau\in [0,T],
  \]
where we have defined, using the Riesz representation theorem, $F(t)\in X$ such that
\begin{equation}\label{defF}
<F(t),u'>_X=\int_{\Omega} f(t)u'+\int_{\partial\Omega} g(t)u', \quad \forall u'\in X.
\end{equation}
\end{definition}

\begin{proposition}\label{NH}
  Let $T > 0$, $u_0 \in X$, $f\in L^2((0,T)\times\Omega)$, $g_0\in L^2((0,T)\times\Gamma_0)$ and $g_1\in L^2((0,T)\times\Gamma_1)$. Then
  $(S_{f,g_0,g_1})^n$ possesses a unique solution $u$.
% Moreover, one has the following
%   estimate
%   \begin{equation} \label{estimatione}\|u \|_{L^{\infty} ((0, T), X)} \leqslant C (T,\varepsilon) \left(
%      \|u_0 \|_X + \|f\|_{L^2 ((0, T) \times \Omega)}+\|g\|_{L^2 ((0, T) \times \partial\Omega)} \right),  \end{equation}
%   with $C (T,\varepsilon)$ only depending on $T$ and $\varepsilon$.
\end{proposition}

\begin{proof}

Let us first assume that $u$ belongs to  $u\in \mathcal{C} ([0, T], {\cal D(A)}) \cap \mathcal{C}^1 ([0, T ], X)$. Then, it is easy to prove that $u$ is a solution to $(S_{f,g_0,g_1})^n$ if and only if, for any $\phi \in {\cal D(A^*)}$,
$$<u_t, \phi>_X=<u,{\cal A^*} \phi>_X + <F,\phi>_X \text{ in } [0,T].$$
Consequently, thanks to Duhamel formula, $u$ is a solution of $(S_{f,g_0,g_1})^n$ is equivalent to
\[  u(t) = e^{t \mathcal{A}}u_0 + \int_0^t e^{(t - s) \mathcal{A}} F (s)ds, \quad \forall t\in[0,T],\]
since ${\cal D}({\cal A}^*)$ is dense in $X$.

%For a general solution $u \in \mathcal{C} ([0, T], X)$, we approximate it in $L^{\infty}((0,T),X)$ by some function $u_n\in \mathcal{C} ([0, T], {\cal D(A)}) \cap \mathcal{C}^1 ([0, T ], X)$. Using the Definition 4, we get that $F_n:= (u_n)_t- {\cal A}u_n$ satisfies, for any $t\in [0,T]$,
%$$ \int_0^t <F_n,\varphi>_X \tendvers{n}{\infty} \int_0^t <F,\varphi>_X  \quad \forall \varphi \in \mathcal{C}( [0,t], {\cal D(A^*)}) \cap \mathcal{C}^1([0,t],X) $$
%so that, for any $\phi \in {\cal D(A^*)}$,
%  \begin{eqnarray*} \forall t\in[0,T], \quad <u(t),\phi>_X &=& <e^{t \mathcal{A}}u_0,\phi>_X + \lim_{n \rightarrow \infty} \int_0^t < F_n (s),e^{(t - s) \mathcal{A^*}} \phi>_X ds\\
%&=& <e^{t \mathcal{A}}u_0,\phi>_X +  \int_0^t < F (s),e^{(t - s) \mathcal{A^*}} \phi>_X ds\\
%&=& <e^{t \mathcal{A}}u_0,\phi>_X +  \int_0^t < e^{(t - s) \mathcal{A}}F (s), \phi>_X ds.
%\end{eqnarray*}
%Using the density of ${\cal D(A^*)}$ in $X$, the proof is now complete.

In order to prove the existence of a solution $u\in {\cal C}([0,T],X)$, we approximate $(u_0,f,g)$ by
$$
(u_0^p,f^p,g_0^p,g_1^p)\in {\cal D}({\cal A})\times {\cal C}([0,T],L^2(\Omega))\times {\cal C}([0,T],L^2(\Gamma_0))\times {\cal C}([0,T],L^2(\Gamma_1))
$$
in
$$
X\times L^2((0,T),L^2(\Omega))\times L^2((0,T),L^2(\Gamma_0))\times L^2((0,T),L^2(\Gamma_1)).
$$
Then, we know that $u^p$ given by
$$
u^p(t) = e^{t \mathcal{A}}u_0^p + \int_0^t e^{(t - s) \mathcal{A}} F^p (s)ds, \quad \forall t\in[0,T],
$$
where $F^p$ is related to $(f^p,g_0^p,g_1^p)$ by formula (\ref{defF}), is the solution of $(S_{f^p,g_0^p,g_1^p})^n$. Passing to the limit in this identity, we get that $u^p\rightarrow u$ in ${\cal C}([0,T],X)$ with
$$
u(t): = e^{t \mathcal{A}}u_0 + \int_0^t e^{(t - s) \mathcal{A}} F (s)ds,\quad \forall t\in[0,T].
$$
Furthermore, passing to the limit in Definition~\ref{Ladefinition} for $u^p$, we obtain that $u$ is a solution of $(S_{f,g_0,g_1})^n$. Moreover, this argument also proves the uniqueness of the solution.
\end{proof}

We now focus on some regularization effect for our general system and for some technical reasons we want to have some explicit dependence of the bounds on $\varepsilon$. For this result, it is essential that $\varepsilon>0$.

\begin{lemma}\label{RE}
Let $\varepsilon\in (0,1)$, $f\in L^2((0,T)\times\Omega)$, $g_0\in L^2((0,T)\times\Gamma_0)$ and $g_1\in L^2((0,T)\times\Gamma_1)$.
\begin{itemize}
\item Let $u_0\in X$.  Then if $u$ is the solution of
$(S_{f,g_0,g_1})^n$, $u$ belongs to $L^2((0,T),H^1(\Omega))\cap {\cal C}([0,T],X)$ and
\begin{equation}\label{L2H1}
\varepsilon^{1/2}\|u\|_{L^2((0,T),H^1(\Omega))}+\|u\|_{L^{\infty}((0,T),X)} \leq C(\|u_0\|_{X}+\|f\|_{L^2((0,T)\times\Omega)}+\|g\|_{L^2((0,T)\times\partial\Omega)})
\end{equation}
for some $C>0$ only depending on $T$.
\item Let now $u_0\in H^1(\Omega)$. Then $u\in L^2((0,T),\mathcal{D(A)})\cap {\cal C}([0,T],H^1(\Omega))$, $u_t \in L^2((0,T),X)$ and
\begin{equation}\label{L2H2}
\begin{array}{l}\displaystyle
\varepsilon^{1/2}\|u\|_{L^\infty([0,T],H^1(\Omega))}+\varepsilon(\|\Delta u\|_{L^2((0,T) \times \Omega)}+\| \partial_\nu u \|_{L^2((0,T)\times \partial \Omega)})+\|u_t\|_{L^2([0,T],X)}
\\ \noalign{\medskip}\displaystyle
\leq \frac{C}{\varepsilon^{1/2}}(\|u_0\|_{H^1(\Omega)}+\|f\|_{L^2((0,T)\times\Omega)}+\|g\|_{L^2((0,T)\times\partial\Omega)}),
\end{array}
\end{equation}
where $C>0$ only depends on $T$.
\end{itemize}
\end{lemma}
\begin{proof}
Approximating $u$ by a regular function (in $L^2((0,T),{\cal D}(\overline{\Omega}))$), all computations below are justified.
\begin{itemize}
  \item First, we multiply our main equation by $u$ and we integrate on $\Omega$. An application of Green-Riemann formula with use of the boundary conditions gives the identity
      \[  \frac{1}{2}\frac{d}{d t} \| u\|^2_X + \frac{1}{2}\int_{\partial\Omega} |u|^2 + \varepsilon \int_{\Omega} | \nabla u |^2 \ = \int_{\Omega} f u + \int_{\partial \Omega} g  u,\]
      which immediately yields
      \begin{equation*}
      \frac{d}{d t} \| u\|^2_X +\int_{\partial\Omega} |u|^2 + 2\varepsilon \int_{\Omega} | \nabla u |^2 \leq \| u\|^2_X +\|f\|^2_{L^2(\Omega)}+\|g\|^2_{L^2(\partial\Omega)}.
      \end{equation*}
      Finally, Gronwall's lemma yields (\ref{L2H1}).
  \item Now, we multiply the main equation by $ u_t$ and we integrate
on $\Omega$, which, using  Green-Riemann formula and the boundary conditions provide the
identity
\[  \| u_t\|^2_X + \int_{\Omega} \partial_{x_n} u
   u_t + \frac{\varepsilon}{2} \frac{d}{d t} \left(
   \int_{\Omega} | \nabla u |^2 \right)+\frac{1}{2} \frac{d}{d t} \left(
   \int_{\Gamma_1} | u |^2 \right) = \int_{\Omega} f u_t +\int_{\partial \Omega} g u_t,\]
Now
\[ - \int_{\Omega} u_t\partial_{x_n} u  \leqslant
   \frac{1}{2} \int_{\Omega} | \nabla u |^2 + \frac{1}{2} \| u_t  \|^2_X \]
and, using again Young's inequality, we have
\[  \int_{\Omega} f u_t + \int_{\partial \Omega} g u_t
   \leqslant \frac{1}{4} \| u_t\|^2_X + C (\| f \|^2_{L^2(\Omega)}+\| g \|^2_{L^2(\partial\Omega)}).
   \]
We have thus proven the following inequality
\begin{equation}\label{EI}
 \| u_t \|^2_X +2\varepsilon \frac{d}{d t} \left(
   \int_{\Omega} | \nabla u |^2 \right)+2\frac{d}{d t}\int_{\Gamma_1}|u|^2 \leq 2\int_{\Omega} | \nabla u |^2 +
C \left(\| g \|^2_{L^2(\partial\Omega)}+\| f \|^2_{L^2(\Omega)}\right).
\end{equation}
We integrate here between $0$ and $t$ and estimate the first term in the right-hand side using (\ref{L2H1})
to get
\begin{equation*}%\label{CH1}
\int_{\Omega} | \nabla u |^2 \leqslant
   \frac{C}{\varepsilon^2} \left(\int_0^T (\| f(t) \|^2_{L^2(\Omega)}+\| g(t) \|^2_{L^2(\partial\Omega)}) d
   t+\| u_0 \|^2_{H^1(\Omega)}\right),
\end{equation*}
for $t\in (0,T)$. We now inject this into \eqref{EI} to get the second part of the required result
\[\| u_t \|^2_{L^2((0,T);X)} \leqslant \frac{C}{\varepsilon} \left(\int_0^T (\| f(t) \|^2_{L^2(\Omega)}+\| g(t) \|^2_{L^2(\partial\Omega)}) d
   t+\| u_0 \|^2_{H^1(\Omega)}\right). \]
Writing the problem as
\begin{eqnarray*}
   \left\{  \begin{array}{l}
     -\Delta u = \frac{1}{\varepsilon}(f - u_t-\partial_{x_n} u)\\
    \partial_{\nu} u = \frac{g_0}{\varepsilon} - u_t\\
    \partial_{\nu} u = \frac{g_1}{\varepsilon} - u_t-\frac{1}{\varepsilon}u
  \end{array} \right.
\begin{array}{c}
     \text{in } (0, T) \times \Omega,\\
     \text{on } (0, T) \times \Gamma_0,\\
     \text{on } (0, T) \times \Gamma_1,
   \end{array}
\end{eqnarray*}
we get that $u \in L^2 ((0, T), \mathcal{D(A)})$ and the existence of some constant $C>0$ such that
\begin{eqnarray*}
& &\|\Delta u\|_{L^2((0,T) \times \Omega)}+\| \partial_\nu u \|_{L^2((0,T)\times \partial \Omega)}\\
& & \leq \frac{C}{\varepsilon} (\|f\|_{L^2((0,T)\times\Omega)}+\|g\|_{L^2((0,T)\times\partial\Omega)}+ \|u\|_{L^2 ((0, T), H^1 (\Omega))}+\|  u_t \|_{L^2((0,T),X)})
\end{eqnarray*}

which gives the last part of the required result (\ref{L2H2}).
\end{itemize}
\end{proof}

\begin{remark}\label{remarqueH2}
Working on the non-homogeneous adjoint problem and using the backward semigroup $e^{(T-t)\mathcal{A}^{\ast}}$, one is able to show that the estimates of Lemma \ref{RE} hold for the system
\[ (S'_{f,g_0,g_1})^n \left\{ \begin{array}{c}
     \varphi_t + \partial_{x_n} \varphi + \varepsilon \Delta \varphi = f\\
     \varepsilon (\varphi_t - \partial_{\nu} \varphi) - \varphi = g_0\\
     \varepsilon(\varphi_t - \partial_{\nu} \varphi) = g_1\\
     \varphi (T, .) = \varphi_T
   \end{array} \right. \begin{array}{c}
     \text{in } (0, T) \times \Omega,\\
     \text{on } (0, T) \times \Gamma_0,\\
     \text{on } (0, T) \times \Gamma_1,\\
     \text{in } \Omega.
   \end{array} \]
\end{remark}

\subsection{Backward uniqueness}

We will show here the backward uniqueness of systems $(S_0)^n$ and $(S')^n$ for $\varepsilon>0$, thanks to
the well-known result of Lions-Malgrange (\cite{LM}) and the regularization effect.
\begin{lemma}\label{Reg}
Let $u_0\in X$ and $\delta \in (0,T)$. Then, the solution $u$ of $(S_0)^n$ satisfies $u\in \mathcal{C}([\delta,T],\mathcal{ D(A) })$.
\end{lemma}
\begin{proof}
Our strategy is to show that $u \in L^2 ((\delta, T) ; H^2 (\Omega))$ and
 $ u_t \in L^2 ((\delta, T) ; H^2 (\Omega))$. This easily implies
that $u \in \mathcal{C}([\delta, T] ; H^2 (\Omega))$.
\par
We first select some regular cut-off function $\theta_1$ such that $\theta_1=1$ on $(\delta/2,T)$ and $\theta_1=0$ on $(0,\delta/4)$. If $u_1=\theta_1 u$, we have that
\begin{eqnarray*}
  \left\{ \begin{array}{l}
    u_{1,t} + \partial_{x_n} u_1 - \varepsilon \Delta u_1 = \theta_1' u\\
    \varepsilon(u_{1,t} + \partial_{\nu} u_1) = \varepsilon\theta_1' u\\
    \varepsilon (u_{1,t} + \partial_{\nu} u_1) + u_1 = \varepsilon
    \theta_1' u\\
u_1(0,.)=0
  \end{array} \right.
\begin{array}{c}
     \text{in } (0, T) \times \Omega,\\
     \text{on } (0, T) \times \Gamma_0,\\
     \text{on } (0, T) \times \Gamma_1,\\
      \text{in }  \Omega,
   \end{array}
\end{eqnarray*}
that is $u_1$ satisfies $(S_{\theta_1' u,\varepsilon\theta_1' u,\varepsilon\theta_1' u})^n$. An application of Lemma \ref{RE} gives that $u_1\in L^2((0,T),\mathcal{ D(A) })$ and $u_{1,t}\in L^2((0,T),X)$. This implies in particular that $u\in L^2((\delta,T),\mathcal{ D(A) })$.
\par
We now focus on $u_t$ and we select another cut-off function $\theta_2$ such that $\theta_2=1$ on $(\delta,T)$ and $\theta_2=0$ on $(0,\delta/2)$. We write the system satisfied by $\theta_2 u$ and we differentiate it with respect to time.
One deduces that $u_2=(\theta_2 u)_t$ is the solution of
$$(S_{\theta_2' u_t+\theta_2'' u,\varepsilon(\theta_2' u_t+\theta_2''u),\varepsilon(\theta_2' u_t+\theta_2'' u)})^n.$$
Using that $|\theta_2'| \lesssim |\theta_1|$, Lemma \ref{RE} gives that $u_2\in L^2((0,T),\mathcal{ D(A) })$ that is $u_t\in L^2((\delta,T),\mathcal{ D(A) })$.

\end{proof}

\begin{proposition}\label{BU}
  Assume that $u$ is a weak solution of $(S_0)^n$ such that $u (T) = 0$. Then
  $u_0 \equiv 0$.
\end{proposition}

\begin{proof}
If $0 < \delta < T$, then Lemma \ref{Reg} shows that  $u (\delta) \in \mathcal{ D(A) }$.
We will now apply Th\'eor\`eme 1.1 of \cite{LM} to $u$ as a solution of $(S_0)^n$ in the time interval $(\delta,T)$ to show that $u
  (\delta,.) = 0$. The bilinear form $a$ defined in \eqref{a-def} can be split into two
bilinear forms on $V$ defined by
  \[ a (u_1, u_2) = a_0 (t, u_1, u_2) + a_1 (t, u_1, u_2) \]
with
  \[ a_0 (t, u_1, u_2) = \varepsilon\int_{\Omega} \nabla u_1. \nabla u_2, \quad a_1 (t, u_1, u_2) =
     \int_{\Omega} \partial_{x_n} u_1 u_2 +\int_{\Gamma_1} u_1 u_2 . \]
The first four hypotheses in \cite{LM} (see (1.1)-(1.4) in that reference) are satisfied since $u\in L^2(\delta,T;V)\cap H^1(\delta,T;H)$ (from Lemma~\ref{RE} above),
$u(t)\in \mathcal{D}(\mathcal{A})$ for almost every $t\in (\delta,T)$, $u_t=\mathcal{A}u$ and $u(T,.)\equiv0$.
\noindent On the other hand, it is clear that $a_0$ and $a_1$ are continuous bilinear forms on $V$ and
that they do not depend on time $t$. It is also straightforward to see that
\[ \forall u \in V ,\, a_0 (t, u, u) +\|u \|^2_X \geqslant \min\{1,\varepsilon\} \|u \|^2 \]
and, for some constant $C > 0$,
\[  \forall u,v \in V ,\, |a_1 (t, u_1, u_2) | \leqslant C\|u_1 \| \|u_2 \|_X . \]
This means that Hypothesis I in that reference is fulfilled. We have shown that $u(\delta,.) = 0$, which finishes the proof
  since
\begin{equation*}
 u_0 = \lim_{\delta \rightarrow 0} u (\delta) = 0.
\end{equation*} \end{proof}

\begin{remark} The result also holds for the adjoint system:
  if $\varphi$ is a weak solution of $(S')^n$ such that $\varphi (0) \equiv
  0$ then $\varphi_T = 0$. The proof is very similar to the one above and is left to the reader.
\end{remark}

\subsection{Proof of Proposition \ref{OBS}}
The fact that $u=0$ on $(0,T)\times\Gamma_1$ implies $u_0\equiv 0$ is a straightforward consequence of the backward uniqueness
of system $(S_0)^n$. Indeed, similarly as (\ref{OI}), we can prove the observability inequality
$$
\|u(T,.)\|_{X}\lesssim \|u\|_{L^2((0,T)\times\Gamma_1)},
$$
(see Remark~\ref{Betis} below), which combined with Proposition~\ref{BU} gives $u_0\equiv 0$.
\par \noindent To show that \eqref{FE} is false, we first need to prove some well-posedness result of $(S_0)^n$ with $u_0$ in a less regular space.
For $s\in (0,1/2)$, we define $X^s$ as the closure of $\mathcal{D} (
\bar{\Omega}$) for the norm $$\|u\|_{X^s} \assign \left( \|u\|_{H^s(\Omega)}^2 +
\varepsilon \|u\|_{H^s(\partial\Omega)}^2 \right)^{\frac{1}{2}},$$ where $H^s(\Omega)$ (resp. $H^s(\partial\Omega)$) stands for the usual Sobolev space on $\Omega$ (resp. $\partial\Omega$). One easily shows that $X^s$ is a Hilbert space. We now denote $X^{-s}$ the set of functions $u_2$ such that the linear form
$$u_1\in \mathcal{D} (
\bar{\Omega}) \longmapsto<u_1,u_2>_X$$ can be extended in a continuous way on $X^s$. If $u_1\in X^{-s}$, $u_2\in X^s$ we denote this extension by $<u_1,u_2>_{-s,s}$.
\par
\vskip 0.2cm
If $u_0\in X^{-s}$, we say that $u$ is a \textbf{solution by transposition} of $(S_0)^n$ if, for every $f\in L^2((0,T)\times \Omega)$
\begin{equation*}
\int_0^T \int_\Omega uf+<u_0,\varphi(0,.)>_{-s,s}=0,
\end{equation*}
where $\varphi$ is the weak solution (see Remark 4) of
\[ (S'_{f,0,0})^n \left\{ \begin{array}{c}
     \varphi_t + \partial_{x_n} \varphi + \varepsilon \Delta \varphi = f\\
     \varepsilon (\varphi_t - \partial_{\nu} \varphi) - \varphi = 0\\
     \varphi_t - \partial_{\nu} \varphi = 0\\
     \varphi (T, .) = 0
   \end{array} \right. \begin{array}{c}
     \text{in } (0, T) \times \Omega,\\
     \text{on } (0, T) \times \Gamma_0,\\
     \text{on } (0, T) \times \Gamma_1,\\
     \text{in } \Omega.
   \end{array} \]
Using the Riesz representation theorem and continuity (see Remark 4)  of
$$f\in L^2((0,T)\times \Omega)\longmapsto \varphi(0,.)\in H^1(\Omega),$$
it is obvious that, for any $u_0\in X^{-s}$, there exists a solution by transposition of $(S_0)^n$, $u\in L^2((0,T)\times\Omega)$, such that
\begin{equation*}
\|u\|_{L^2((0,T)\times\Omega)} \lesssim \|u_0\|_{X^{-s}}.
\end{equation*}
Moreover, if $u_0 \in X=X^0$, Lemma \ref{RE} shows that the solution $u$ satisfies
\begin{equation*}
\|u\|_{L^2((0,T), H^1( \Omega))} \lesssim \|u_0\|_{X^0}.
\end{equation*}
If one uses classical interpolation results (see \cite{LMa}), we can deduce that for every $\theta\in (0,1/2)$, for any $u_0\in X^{-s\theta}=[X,X^{-s}]_{1-\theta}$,
there exists a solution $u\in L^2((0,T), H^{1-\theta}( \Omega))$ such that
\begin{equation*}
\|u\|_{L^2((0,T), H^{1-\theta}( \Omega))} \lesssim \|u_0\|_{X^{-s\theta}}.
\end{equation*}
Using classical trace result, we have that
\begin{equation}\label{I1}
\|u\|_{L^2((0,T)\times \partial\Omega)} \lesssim \|u_0\|_{X^{-s\theta}}.
\end{equation}
\par \noindent On the other hand, estimate \eqref{FE} implies in particular that
\begin{equation}\label{I2}
\|u_0 \|_X \lesssim \|u\|_{L^2 ( (0, T)\times\Gamma_1 )}\quad \forall u_0 \in \mathcal{D}(\overline{\Omega}).
\end{equation}
Finally, \eqref{I1} and \eqref{I2} yields the contradictory inclusion $X^{-s\theta}\hookrightarrow X$. $\square$

\section{Carleman inequality in dimension 1}

In this paragraph, we will establish a Carleman-type inequality keeping track of the explicit dependence of all the constants with respect to $T$ and $\varepsilon$. As in \cite{FI}, we introduce the following weight functions:
\[ \forall x\in [-L,0], \quad \eta(x) \assign 2 L + x,\,\,\,\, \alpha(t,x) \assign \frac{\lambda - e^{\eta(x)}}{\varepsilon^2t ( T - t)}, \,\,\,\,\phi(t,x)
   \assign \frac{e^{ \eta(x)}}{\varepsilon^2 t ({T} - t)} , \]
where $\lambda>e^{2L}$.
\par \noindent The rest of this paragraph will be dedicated to the proof of the following inequality:
\begin{theorem}\label{The:Carleman}
  There exists $C>0$ and $s_0>0$ such that for every $\varepsilon\in (0,1)$ and every $s \geqslant s_0 ( \varepsilon T + \varepsilon^2T^2)$ the following inequality is satisfied for every $\varphi_T\in X$:
\begin{equation}\label{Carlemana}
\begin{array}{l}\displaystyle
 s^3 \int_{(0,T)\times (-L,0)} \phi^3
     e^{- 2 s \alpha} | \varphi |^2 + s \int_{(0,T)\times (-L,0)} \phi e^{- 2 s \alpha} | \varphi_x
     |^2 + s^3 \int_{(0, T)\times \{0,-L\}} \phi^3 e^{- 2 s \alpha} | \varphi
     |^2
     \\ \displaystyle
\leqslant C s^7\int_{(0, T)\times \{0\}}e^{-4s\alpha+2s\alpha(.,-L)}\phi^7|\varphi|^2.
     \end{array}
\end{equation}
Here, $\varphi$ stands for the solution of $(S')$ associated to $\varphi_T$.
\end{theorem}

\begin{remark}\label{Betis}
One can in fact obtain the following Carleman estimate with control term in $\Gamma_1$
\begin{equation*}
\begin{array}{l}\displaystyle
 s^3 \int_{(0,T)\times (-L,0)}\phi^3
     e^{- 2 s \alpha} | \varphi |^2 + s \int_{(0,T)\times (-L,0)} \phi e^{- 2 s \alpha} | \varphi_x
     |^2 + s^3 \int_{(0, T)\times \{0,-L\}} \phi^3 e^{- 2 s \alpha} | \varphi
     |^2
     \\ \noalign{\medskip}\displaystyle
\leqslant C s^7\int_{(0, T)\times \{-L\}}e^{-4s\alpha+2s\alpha(.,0)}\phi^7|\varphi|^2.
     \end{array}
\end{equation*}
simply by choosing the weight function $\eta(x)$ equal to $x\mapsto -x+L$. The proof is very similar and the sequel will show how to deal with this case too.
\end{remark}

In order to perform a Carleman inequality for system $(S')$, we first do a scaling in time.
Introducing $\tilde{T} : = \varepsilon T$,
$\tilde{\varphi} (t, x) : = \varphi (t/\varepsilon, x)$ and the weights $\tilde\alpha(t,x):=\alpha(t/\varepsilon,x)$, $\tilde\phi(t,x):=\phi(t/\varepsilon,x)$; we have the following system
\begin{equation}\label{tildevarphi}
\left\{
\begin{array}{l}
       \tilde{\varphi}_t +{\varepsilon}^{-1}\tilde{\varphi}_x + \tilde{\varphi}_{x x} = 0 \\
       \varepsilon^2(\tilde{\varphi}_t - \varepsilon^{-1}\partial_{\nu} \tilde{\varphi})-\tilde\varphi = 0   \\
       \tilde{\varphi}_t - \varepsilon^{-1}\partial_{\nu} \tilde{\varphi} = 0 \\
       \tilde{\varphi}_{|t=\tilde{T}} = \varphi_T
\end{array}
\right.
\left.
\begin{array}{l}
\text{in } q,\\
\text{on } \sigma_0,\\
\text{on } \sigma_1,\\
\text{in } (- L, 0).
\end{array}
\right.
\end{equation}
if $q := (0, \tilde{T})\times (-L,0)$ , $\sigma:=(0, \tilde{T})\times \{-L,0\}$, $\sigma_0:=(0, \tilde{T})\times \{0\}$ and $\sigma_1:=(0, \tilde{T})\times \{-L\}$. We will now explain how to get the following result.
\begin{proposition}\label{Th:Carleman}
  There exists $C>0$, such that for every $\varepsilon\in (0,1)$ and every $s \geqslant C ( \tilde{T} + \varepsilon^{-1}\tilde T^2+\varepsilon^{1/3}\tilde T^{2/3})$,
the following inequality is satisfied for every solution of (\ref{tildevarphi}) associated to $\tilde\varphi_T\in X$:
\begin{equation}\label{Carleman}
\begin{array}{l}\displaystyle
 s^3 \int_q \tilde\phi^3
     e^{- 2 s \tilde\alpha} | \tilde{\varphi} |^2 + s \int_q \tilde\phi e^{- 2 s \tilde\alpha} | \tilde{\varphi}_x
     |^2 + s^3 \int_{\sigma} \tilde\phi^3 e^{- 2 s \tilde\alpha} | \tilde{\varphi}
     |^2\leqslant C s^7\int_{\sigma_0} e^{-4s\tilde\alpha+2s\tilde\alpha(.,-L)}\tilde\phi^7|\tilde\varphi|^2.
     \end{array}
\end{equation}
\end{proposition}
Observe that Proposition~\ref{Th:Carleman} directly implies Theorem~\ref{The:Carleman}.

All along the proof we will need several properties of the weight functions: \ \

\begin{lemma}\label{prop:poids}
{\tmdummy}
  \begin{itemize}
    \item $| \tilde\alpha_t | \lesssim \tilde{T} \tilde\phi^2, \,\,| \tilde\alpha_{x t} | \lesssim
     \tilde{T} \tilde\phi^2, \,\,| \tilde\alpha_{t t} | \lesssim \tilde{T}^2 \tilde\phi^3,$

    \item $\tilde\alpha_x = -  \tilde\phi, \,\,\tilde\alpha_{x x} = - \tilde\phi .$
  \end{itemize}
\end{lemma}

There are essentially two steps in this proof:
\begin{itemize}
\item[-] The first one consists in doing a Carleman estimate similar to that of \cite{FI}. The idea is to compute the $L^2$ product of some operators applied to $e^{-s\tilde\alpha}\tilde\varphi$, do integrations by parts and make a convenient choice of the parameter $s$. This will give an estimate where we still have a boundary term on $\sigma_0$ in the right-hand side.

\item[-] In the second one, we will study the boundary terms appearing in the right hand side due to the boundary conditions. We will perform the proof of this theorem for smooth solutions, so that the general proof follows from a density argument.
\end{itemize}
The big difference with general parabolic equations is the special boundary condition. Up to our knowledge, no Carleman inequalities have been performed when the partial time derivative appears on the boundary condition. The fact that we still have a Carleman inequality for our system comes from the fact that our system remains somehow parabolic.

\vskip0.3cm

We will first estimate the left hand side terms of \eqref{Carleman} like in the classical Carleman estimate (see \cite{FI}). We obtain:
\begin{proposition}\label{Carleman classique}
There exists $C>0$, such that for every $\varepsilon\in (0,1)$ and every $s \geqslant C ( \tilde{T} + \varepsilon^{-1}\tilde T^2+\varepsilon^{1/3}\tilde T^{2/3})$,
the following inequality is satisfied for every solution of (\ref{tildevarphi}) associated to $\tilde\varphi_T\in X$:
\begin{equation}
\begin{array}{l}\displaystyle
s^3\int_q \tilde\phi^3e^{-2s\tilde\alpha}|\tilde\varphi|^2+s^2\int_q \tilde\phi e^{-2s\tilde\alpha}|\tilde\varphi_x|^2
+s^{-1}\int_q \tilde\phi^{-1}e^{-2s\tilde\alpha}(|\tilde\varphi_{xx}|^2+|\tilde\varphi_t|^2)
   \\ \noalign{\medskip}\displaystyle
+ s^3 \int_{\sigma_1}\tilde\phi^3e^{-2s\tilde\alpha} | \tilde\varphi |^2 \lesssim s^5\int_{\sigma_0} \tilde\phi^5e^{-2s\tilde\alpha}|\tilde\varphi|^2+\varepsilon^2 s \int_{\sigma_0} \tilde\phi e^{-2s\tilde\alpha} | \tilde\varphi_t |^2.
\end{array}
\end{equation}
\end{proposition}
\begin{proof}

Let us introduce $\psi \assign \tilde{\varphi} e^{- s \tilde\alpha}$. We state the equations satisfied by $\psi$. In $(0,\tilde T)\times (-L,0)$, we have the identity
\[ P_1 \psi + P_2 \psi = P_3 \psi, \]
where
\begin{equation}\label{P1}
P_1 \psi = \psi_t + 2 s \tilde\alpha_x \psi_x+\varepsilon^{-1}\psi_x,
\end{equation}
\begin{equation}\label{P2}
P_2 \psi = \psi_{x x} + s^2 \tilde\alpha^2_x \psi + s \tilde\alpha_t \psi+\varepsilon^{-1}s\tilde\alpha_x\psi,
\end{equation}
and
$$
P_3 \psi= -s \tilde\alpha_{x x} \psi.
$$
On the other hand, the boundary conditions are:
\begin{equation}\label{enzero}
  \varepsilon^2(\psi_t + s \tilde\alpha_t \psi -\varepsilon^{-1}(\psi_x + s \tilde\alpha_x
  \psi)) - \psi = 0\,\,\,\,\text{on }x=0,
\end{equation}
\begin{equation}\label{en-L}
\varepsilon(\psi_t + s \tilde\alpha_t \psi )+ \psi_x + s \tilde\alpha_x \psi = 0\,\,\,\,\text{on }x=-L.
\end{equation}
We take the $L^2$ norm in both sides of the identity in $q$:
\begin{equation}\label{L2}
\|P_1\psi\|^2_{L^2(q)}+\|P_2\psi\|^2_{L^2(q)}+2(P_1\psi,P_2\psi)_{L^2(q)}=\|P_3\psi\|^2_{L^2(q)}.
\end{equation}
Using Lemma~\ref{prop:poids}, we directly obtain
\[ \left. \left. \right\| P_3 \psi \right\|^2_{L^2 (q)} \lesssim
    s^2
   \int_q \tilde\phi^2 | \psi |^2. \]

We focus on the expression of the double product $(P_1 \psi, P_2 \psi)_{L^2 (q)}$. This product contains 12 terms which will be denoted by $T_{i j} (\psi)$ for $1 \leqslant i \leqslant 3$, $1 \leqslant j \leqslant 4$. We study them successively.
\begin{itemize}
  \item An integration by parts in space and then in time shows that, since $\psi (0, .) = \psi ( \tilde{T}, .) = 0,$ we have
  \begin{eqnarray*}
       T_{1 1} (\psi) = \int_q \psi_t
       \psi_{x x} = -\frac{1}{2} \int_q
       ( | \psi_x |^2)_t + \int_{\sigma_0} \psi_x \psi_t -
       \int_{\sigma_1} \psi_x \psi_t &  & \\
       = \int_{\sigma_0} \psi_x \psi_t -
       \int_{\sigma_1} \psi_x \psi_t. &  &
   \end{eqnarray*}
  Now, we use the boundary condition (\ref{enzero}). We have
  \[ \int_{\sigma_0} \psi_x \psi_t \gtrsim \varepsilon\int_{\sigma_0} | \psi_t |^2 -\varepsilon \tilde T^2s\int_{\sigma_0} \tilde\phi^3| \psi |^2      -\tilde T s\int_{\sigma_0} \tilde\phi^2|\psi|^2. \]
  Thanks to the fact that $\psi (0, .) = \psi ( \tilde{T}, .) = 0$, the same computations can be done on $x=-L$ using (\ref{en-L}) so we obtain the following for this term:
  $$
  T_{1 1} (\psi) \gtrsim -\varepsilon \tilde T^2s\int_{\sigma} \tilde\phi^3| \psi |^2
     -\tilde Ts\int_{\sigma} \tilde\phi^2| \psi |^2.
  $$
  \item Integrating by parts in time, we find
  \[ T_{1 2} (\psi) = \frac{s^2}{2} \int_q \tilde\alpha_x^2 ( | \psi |^2)_t = -s^2 \int_q \tilde\alpha_x \tilde\alpha_{x t} | \psi |^2, \]
     $$
     T_{1 3} (\psi) = \frac{s}{2} \int_q \tilde\alpha_t ( | \psi |^2)_t = -\frac{s}{2} \int_q \tilde\alpha_{tt} | \psi |^2,
     $$
  \[ T_{1 4} (\psi) = \varepsilon^{-1}\frac{s}{2} \int_q \tilde\alpha_x ( | \psi |^2)_t = - \varepsilon^{-1}\frac{s}{2}
     \int_q \tilde\alpha_{t x} | \psi |^2, \]
  and using Lemma~\ref{prop:poids}, we get
  \[ T_{1 2} (\psi) \gtrsim -  \tilde{T}s^2 \int_q \tilde\phi^3 | \psi |^2, \]
     $$
     T_{1 3} (\psi) \gtrsim - \tilde{T}^2s\int_q \tilde\phi^3 | \psi |^2,
     $$
  \[ T_{1 4} (\psi) \gtrsim -  \varepsilon^{-1}\tilde{T}s \int_q \tilde\phi^2 | \psi |^2 . \]
  \item Now, integrating by parts in space, we have
  \[ T_{2 1} (\psi) = s \int_q \tilde\alpha_x ( | \psi_x |^2)_x = - s \int_q \tilde\alpha_{x x} | \psi_x |^2 + s \int_{\sigma_0}
        \tilde\alpha_x | \psi_x |^2 - s \int_{\sigma_1} \tilde\alpha_x | \psi_x
        |^2 . \]
  Thanks to Lemma~\ref{prop:poids} (the choice of $\eta$ is important here), the last term is positive. Using the boundary condition (\ref{enzero}) and Lemma~\ref{prop:poids}, we finally get
  $$
  \begin{array}{l}\displaystyle
  T_{2 1} (\psi) \gtrsim s \int_q \tilde\phi | \psi_x |^2 - s^3 \int_{\sigma_0}
     \tilde\phi^3 | \psi |^2 - \varepsilon^2s^3  \tilde{T}^2 \int_{\sigma_0}\tilde\phi^5 | \psi |^2
       + s\int_{\sigma_1} \tilde\phi | \psi_x |^2- \varepsilon^2s  \int_{\sigma_0} \tilde\phi | \psi_t |^2.
  \end{array}
  $$
  \item An integration by parts in space provides
  $$
  \begin{array}{l}\displaystyle
  T_{2 2} (\psi) = s^3 \int_q
     \tilde\alpha_x^3 ( | \psi |^2)_x = - 3 s^3 \int_q \tilde\alpha_x^2 \tilde\alpha_{x x} | \psi |^2 + s^3
     \int_{\sigma_0} \tilde\alpha^3_x | \psi |^2 - s^3 \int_{\sigma_1}
     \tilde\alpha_x^3 | \psi |^2.
  \end{array}
  $$
  This readily yields
  $$
  T_{2 2} (\psi) \gtrsim s^3 \int_q \tilde\phi^3 | \psi |^2 + s^3 \int_{\sigma_1}
     \tilde\phi^3 | \psi |^2 - s^3
     \int_{\sigma_0} \tilde\phi^3 | \psi |^2.
  $$
  \item Again an integration by parts in space gives
  $$
  \begin{array}{l}\displaystyle
  T_{2 3} (\psi) = s^2 \int_q
     \tilde\alpha_t \tilde\alpha_x ( | \psi |^2)_x
= - s^2 \int_q (\tilde\alpha_t \tilde\alpha_x)_x | \psi |^2 + s^2
     \int_{\sigma_0}  \tilde\alpha_t \tilde\alpha_x | \psi |^2
      - s^2     \int_{\sigma_1} \tilde\alpha_t \tilde\alpha_x | \psi |^2.
  \end{array}
  $$
  Using Lemma~\ref{prop:poids}, we obtain
  \[ T_{2 3} (\psi) \gtrsim - s^2 \tilde{T} \int_q \tilde\phi^3 | \psi |^2 - s^2  \tilde{T}
     \int_{\sigma_0} \tilde\phi^3 | \psi |^2 - s^2  \tilde{T}
     \int_{\sigma_1} \tilde\phi^3 | \psi |^2 . \]
     \item The last integral concerning the second term in the expression of $P_1\psi$ is
     $$
     T_{24}(\psi)=\varepsilon^{-1}s^2 \int_q \tilde\alpha_x^2(|\psi|^2)_x.
     $$
     After an integration by parts in space, we get
     $$
     T_{24}(\psi)\geq -2\varepsilon^{-1}s^2\int_q \tilde\phi^2|\psi|^2-
     \varepsilon^{-1}s^2\int_{\sigma_1} \tilde\phi^2|\psi|^2.
     $$
     \item We consider now the third term in the expression of $P_1\psi$. We have
     $$
     T_{31}(\psi)=\frac{\varepsilon^{-1}}{2}\int_q (|\psi_x|^2)_x\geq -\frac{\varepsilon^{-1}}{2}\int_{\sigma_1}|\psi_x|^2.
     $$

     \item Now, we integrate by parts with respect to $x$ and we have

     \begin{eqnarray*}
     T_{32}(\psi)&=&\frac{\varepsilon^{-1}}{2}s^2\int_q \tilde\alpha_x^2(|\psi|^2)_x\geq-\varepsilon^{-1}s^2\int_q
     \tilde\alpha_x\tilde\alpha_{xx}|\psi|^2-\frac{\varepsilon^{-1}}{2}s^2\int_{\sigma_1}\tilde\alpha_x^2|\psi|^2
     \\      &\geq& -\varepsilon^{-1}s^2\int_q\tilde\phi^2|\psi|^2-\frac{\varepsilon^{-1}}{2}s^2\int_{\sigma_1}\tilde\phi^2|\psi|^2.
     \end{eqnarray*}

     \item Then, using another integration by parts in $x$ we obtain
     $$
     \begin{array}{l}\displaystyle
     T_{33}(\psi)=\frac{\varepsilon^{-1}}{2}s\int_q \tilde\alpha_t(|\psi|^2)_x\gtrsim -\varepsilon^{-1}s\tilde T\int_q \tilde\phi^2|\psi|^2-\varepsilon^{-1}\tilde Ts\int_{\sigma_1}\tilde\phi^2|\psi|^2 -\varepsilon^{-1}\tilde Ts\int_{\sigma_0} \tilde\phi^2|\psi|^2.
     \end{array}
     $$
     \item Finally, arguing as before, we find
     $$
     \begin{array}{l}\displaystyle
     T_{34}(\psi)=\frac{\varepsilon^{-2}}{2}s\int_q \tilde\alpha_x(|\psi|^2)_x\geq \frac{\varepsilon^{-2}}{2}s^2\int_q \tilde\phi|\psi|^2-C\varepsilon^{-2}s\int_{\sigma_0} \tilde\phi|\psi|^2
     \gtrsim -\varepsilon^{-2}s\int_{\sigma_0}\tilde\phi|\psi|^2.
     \end{array}
     $$
\end{itemize}

Putting together all the terms and combining the resulting inequality with (\ref{L2}), we obtain
\begin{equation}\label{L2bis}
\|P_1\psi\|^2_{L^2(q)}+\|P_2\psi\|^2_{L^2(q)}+I_1 (\psi) + I_2 (\psi_x) \lesssim J_1 (\psi) +J_2 (\psi_t) +J_3 (\psi_x)+ L (\psi)
\end{equation}
where the main terms are
$$
I_1 (\psi) = s^3 \int_q   \tilde\phi^3 | \psi |^2+ s^3 \int_{\sigma_1} \tilde\phi^3 | \psi |^2,\quad
   I_2 (\psi_x) = s   \int_q \tilde\phi | \psi_x |^2+s \int_{\sigma_1} \tilde\phi | \psi_x |^2,
$$
the right hand side terms are
$$
\begin{array}{l}\displaystyle
 J_1 (\psi) = (\varepsilon^{-1}\tilde Ts+\varepsilon^{-1}s^2) \int_q \tilde\phi^2 | \psi |^2 + (s^2 \tilde{T} + s
   \tilde{T}^2) \int_q \tilde\phi^3 | \psi
   |^2
   \\ \noalign{\medskip}\displaystyle + (s  \tilde{T}+s^2\varepsilon^{-1}+\varepsilon^{-1}\tilde{T}s) \int_{\sigma_1} \tilde\phi^2 | \psi |^2
  + (s^2  \tilde{T} +\varepsilon s \tilde{T}^2)
   \int_{\sigma_1} \tilde\phi^3 | \psi |^2,
\end{array}
$$
and
$$
J_2 (\psi_t) =\varepsilon^2 s \int_{\sigma_0} \tilde\phi | \psi_t |^2 ,\quad J_3 (\psi_x) =\varepsilon^{-1} \int_{\sigma_1} | \psi_x |^2
$$
and the control terms are
$$
\begin{array}{l}\displaystyle
L (\psi) =\varepsilon^{-2} s\int_{\sigma_0}\tilde\phi|\psi|^2+(s  \tilde{T}+\varepsilon^{-1}\tilde T s )\int_{\sigma_0} \tilde\phi^2 | \psi |^2 +
   ( s^3 + s^2  \tilde{T}+\varepsilon \tilde T^2s) \int_{\sigma_0} \tilde\phi^3 | \psi |^2 +\varepsilon^2s^3\tilde T^2 \int_{\sigma_0} \tilde\phi^5 | \psi |^2.
\end{array}
$$
\vskip0.5cm
Let us now see that we can absorb some right hand side terms with the help of the parameter $s$.
\vskip 0.2 cm $\bullet$ First, we see that the distributed terms in $J_1(\psi)$ can be absorbed by the first term in the definition of $I_1(\psi)$ for a choice of $s\gtrsim \varepsilon^{-1/2}\tilde T^{3/2}+\varepsilon^{-1}\tilde{T}^2$.

\vskip0.1cm $\bullet$ Second, we use the second term of $I_1(\psi)$ in order to absorb the integrals in the second line of the definition of $J_1(\psi)$.
We find that this can be done as long as $s\gtrsim \varepsilon^{-1}\tilde T^2+\varepsilon^{-1/2}\tilde T^{3/2}$ since $\varepsilon<1$.
\vskip0.2cm $\bullet$ Next, we observe that, provided $s\gtrsim \varepsilon^{-1}\tilde T^2$, the second term in $J_2(\psi_x)$ absorbs $J_3(\psi_x)$.

Moreover, we observe that all the control terms can be bounded in the following way:
$$
|L(\psi)|\lesssim s^5\int_{\sigma_0} \tilde\phi^5|\psi|^2
$$
as long as $s\gtrsim \tilde T(1+\varepsilon^{-1}\tilde T)$, using that $\varepsilon<1$.
\vskip0.5cm
 Next, we use the expression of $P_1\psi$ and $P_2\psi$ (see (\ref{P1})-(\ref{P2})) in order to obtain some estimates for the terms $\psi_{t}$ and $\psi_{xx}$ respectively:
$$
\begin{array}{l}\displaystyle
s^{-1}\int_q \tilde\phi^{-1}|\psi_t|^2\lesssim s^2\int_q \tilde\phi|\psi_x|^2
+s^{-1}\varepsilon^{-2}\int_q \tilde\phi^{-1}|\psi_x|^2+s^{-1}\int_q \tilde\phi^{-1}|P_1\psi|^2
\lesssim s^2\int_q \tilde\phi|\psi_x|^2
+\int_q |P_1\psi|^2
\end{array}
$$
for any $s\gtrsim \varepsilon^{-1}\tilde T^2$ and
$$
\begin{array}{l}\displaystyle
s^{-1}\int_q \tilde\phi^{-1}|\psi_{xx}|^2\lesssim s^3\int_0^{\tilde T}\int_{-L}^0\tilde\phi^3|\psi|^2
+s\tilde T^{2}\int_q \tilde\phi^{3}|\psi_x|^2+\varepsilon^{-2}s^2\int_q \tilde\phi|\psi|^2
\\ \noalign{\medskip}\displaystyle
\hskip2cm+s^{-1}\int_q \tilde\phi^{-1}|P_2\psi|^2\lesssim s^3\int_q \tilde\phi^3|\psi|^2+s^2\int_q \tilde\phi|\psi_x|^2
+\int_q |P_2\psi|^2,
\end{array}
$$
for $s\gtrsim \tilde T+\varepsilon^{-1}\tilde T^2$.

\noindent Combining all this with (\ref{L2bis}), we obtain
\begin{equation}\label{CI}
\begin{array}{l}\displaystyle
s\int_q \tilde\phi(s^2\tilde\phi^2|\psi|^2+|\psi_x|^2)
+s^{-1}\int_q \tilde\phi^{-1}(|\psi_{xx}|^2+|\psi_t|^2)+ s^3 \int_{\sigma_1} \tilde\phi^3 | \psi |^2
   \\ \noalign{\medskip}\displaystyle
\lesssim s^5\int_{\sigma_0}\tilde\phi^5|\psi|^2+\varepsilon^2s \int_{\sigma_0} \tilde\phi | \psi_t |^2,
\end{array}
\end{equation}
for $s\gtrsim \tilde T(1+\varepsilon^{-1}\tilde T)$.

\vskip0.5cm
Finally, we come back to our variable $\tilde{\varphi}$. We first remark that $\psi_x=e^{-s\tilde\alpha}(\tilde\varphi_x+s\tilde\phi\tilde\varphi)$ and so
$$
s\int_q \tilde\phi e^{-2s\tilde\alpha}|\tilde\varphi_x|^2\lesssim s\int_q \tilde\phi|\psi_x|^2+s^3\int_q \tilde\phi^3|\psi|^2.
$$
Then, we have that $\psi_t=e^{-s\tilde\alpha}(\tilde\varphi_t-s\tilde\alpha_t\tilde\varphi)$, hence
$$
\begin{array}{l}\displaystyle
s^{-1}\int_q \tilde\phi^{-1}e^{-2s\tilde\alpha}|\tilde\varphi_t|^2\lesssim s^{-1}\int_q \tilde\phi^{-1}|\psi_t|^2+s\tilde T^2\int_q \tilde\phi^3|\psi|^2
\lesssim s^{-1}\int_q \tilde\phi^{-1}|\psi_t|^2+s^3\int_q \tilde\phi^3|\psi|^2
\end{array}
$$
for $s\gtrsim \tilde T$. Analogously, we can prove that
$$
s^{-1}\int_0^{\tilde T}\int_q \tilde\phi^{-1}e^{-2s\tilde\alpha}|\tilde\varphi_{xx}|^2\lesssim s^{-1}\int_q \tilde\phi^{-1}|\psi_{x x}|^2+s^3\int_q \tilde\phi^3|\psi|^2+s^2\int_q \tilde\phi|\psi_x|^2,
$$
for $s\gtrsim \tilde T^2$.

\noindent We combine this with (\ref{CI}) and we obtain the required result
\begin{equation}\label{CI2}
\begin{array}{l}\displaystyle
s^3\int_q \tilde\phi^3e^{-2s\tilde\alpha}|\tilde\varphi|^2+s\int_q \tilde\phi e^{-2s\tilde\alpha}|\tilde\varphi_x|^2
+s^{-1}\int_q \tilde\phi^{-1}e^{-2s\tilde\alpha}(|\tilde\varphi_{xx}|^2+|\tilde\varphi_t|^2)+ s^3 \int_{\sigma_1}\tilde\phi^3e^{-2s\tilde\alpha} | \tilde\varphi |^2
   \\ \noalign{\medskip}\displaystyle
 \lesssim s^5\int_{\sigma_0} \tilde\phi^5e^{-2s\tilde\alpha}|\tilde\varphi|^2+\varepsilon^2 s \int_{\sigma_0} \tilde\phi e^{-2s\tilde\alpha} | \tilde\varphi_t |^2,
\end{array}
\end{equation}
for $s\gtrsim \tilde T(1+\varepsilon^{-1}\tilde T)+\varepsilon^{1/3}\tilde T^{2/3}$ and using that
$$
\varepsilon^2 s^3\tilde T^2 \int_{\sigma_0} \tilde\phi^5 e^{-2s\tilde\alpha} | \tilde\varphi|^2 \lesssim s^5\int_{\sigma_0} \tilde\phi^5e^{-2s\tilde\alpha}|\tilde\varphi|^2
$$
which is true for $s\gtrsim \tilde T$ (recall that $\varepsilon<1$).
\end{proof}

With this result, we will now finish the proof of Proposition~\ref{Th:Carleman}.
\paragraph{Estimate of the boundary term}

In this paragraph we will estimate the boundary term
$$
\varepsilon^2 s \int_{\sigma_0} \tilde\phi e^{-2s\tilde\alpha} | \tilde\varphi_t |^2 .$$
After an integration by parts in time, we get
\begin{equation}\label{France98}
\begin{array}{l}\displaystyle
\varepsilon^2s\int_{\sigma_0} \tilde\phi e^{-2s\tilde\alpha} | \tilde\varphi_t |^2=\frac{\varepsilon^2}{2}s\int_{\sigma_0} (\tilde\phi e^{-2s\tilde\alpha})_{tt} | \tilde\varphi |^2
-\varepsilon^2 s\int_{\sigma_0} \tilde\phi e^{-2s\tilde\alpha} \tilde\varphi\, \tilde\varphi_{tt}
\\ \noalign{\medskip}\displaystyle
\lesssim \tilde{T}^2 \varepsilon^2 s^3\int_{\sigma_0} \tilde\phi^5 e^{-2s\tilde\alpha}|\tilde\varphi|^2+\varepsilon^2s\int_{\sigma_0} \tilde\phi e^{-2s\tilde\alpha} |\tilde\varphi|\, |\tilde\varphi_{tt}|,
\end{array}
\end{equation}
for $s\gtrsim \tilde{T}+\tilde{T}^2$. In order to estimate the second time derivative at $x=0$, we will apply some a priori estimates for the adjoint system.
Indeed, let us consider the following function
$$
\zeta(t,x):=\theta(t)\tilde\varphi_t:=e^{-s\tilde\alpha(t,-L)}\tilde\phi^{-5/2}(t,-L)\tilde\varphi_t(t,x).
$$
Then, this function fulfills the following system
\begin{equation*}
\left\{
\begin{array}{l}
       \zeta_t +{\varepsilon}^{-1}\zeta_x + \zeta_{x x} = \theta_t\tilde\varphi_t \\
       \varepsilon^2(\zeta_t - \varepsilon^{-1}\partial_{\nu} \zeta)-\zeta = \varepsilon^2\theta_t\tilde\varphi_t   \\
       \zeta_t - \varepsilon^{-1}\partial_{\nu} \zeta = \theta_t\tilde\varphi_t \\
       \zeta_{|t=T} = 0
\end{array}
\right.
\left.
\begin{array}{l}
\text{in } (0, \tilde{T})\times (-L,0),\\
\text{on } (0, \tilde{T})\times \{0\},\\
\text{on } (0, \tilde{T})\times \{-L\},\\
\text{in } (- L, 0).
\end{array}
\right.
\end{equation*}
Using Remark~\ref{remarqueH2} for $(t,x) \longmapsto\zeta(\varepsilon t,x)$ we find in particular
$$
\varepsilon\int_{\sigma}|\zeta_{t}|^2\lesssim\int_q |\theta_t|^2|\tilde\varphi_t|^2+\varepsilon^2\int_{\sigma} (\theta_t)^2|\tilde\varphi_t|^2.
$$
This directly implies that
$$
\varepsilon\int_{\sigma} \theta^2|\tilde\varphi_{tt}|^2\lesssim\int_q |\theta_t|^2|\tilde\varphi_t|^2+\varepsilon^2\int_{\sigma} (\theta_t)^2|\tilde\varphi_t|^2.
$$
Integrating by parts in time in the last integral, we have
$$
\begin{array}{l}\displaystyle
\varepsilon\int_{\sigma} \theta^2|\tilde\varphi_{tt}|^2\lesssim \left(\int_q |\theta_t|^2|\tilde\varphi_t|^2+\varepsilon^2\int_{\sigma} ((\theta_t)^2)_{tt}|\tilde\varphi|^2
+\varepsilon^3\int_{\sigma} (\theta_t)^4\theta^{-2}|\tilde\varphi|^2
\right)+\frac{\varepsilon}{2}\int_{\sigma} \theta^2|\tilde\varphi_{tt}|^2.
\end{array}
$$
From the definition of $\theta(t)$ and multiplying the previous inequality by $s^{-3}$, we find that (since $\varepsilon<1$)
\begin{eqnarray}\label{final}
\varepsilon s^{-3}\int_{\sigma}(e^{-2s\tilde\alpha}\tilde\phi^{-5})(.,-L)|\tilde\varphi_{tt}|^2 &\lesssim&
\tilde T^2\left(s^{-1} \int_q \tilde\phi^{-1}e^{-2s\tilde\alpha}|\tilde\varphi_t|^2+\varepsilon^2 \tilde T^2 s\int_{\sigma}(\tilde\phi^3e^{-2s\tilde\alpha})(.,-L) |\tilde\varphi|^2 \right) \nonumber \\
&\lesssim &
\tilde T^2\left(s^{-1} \int_q \tilde\phi^{-1}e^{-2s\tilde\alpha}|\tilde\varphi_t|^2+ s^3\int_{\sigma}(\tilde\phi^3e^{-2s\tilde\alpha})(.,-L) |\tilde\varphi|^2 \right)
\end{eqnarray}
for $s\gtrsim \varepsilon \tilde T$. Using Cauchy-Schwarz inequality in the last term of the right hand side of (\ref{France98}), we obtain
$$
\varepsilon^2 s\int_{\sigma_0} \tilde\phi e^{-2s\tilde\alpha} |\tilde\varphi|\, |\tilde\varphi_{tt}|\lesssim \varepsilon\tilde T^{-2}s^{-3}\int_{\sigma}(e^{-2s\tilde\alpha}\tilde\phi^{-5})(.,-L)|\tilde\varphi_{tt}|^2+\varepsilon^{3}\tilde T^2s^5\int_{\sigma_0}e^{-4s\tilde\alpha+2s\tilde\alpha(.,-L)}\tilde\phi^7|\tilde\varphi|^2
$$
Combining this with (\ref{final}) and (\ref{CI2}) yields the desired inequality (\ref{Carleman}).
\noindent $\square$

\section{Observability and control}

In this section, we prove Theorem~\ref{Seville}.

\subsection{Dissipation and observability result}

Our first goal, as in \cite{CG}, will be to get some dissipation result. Even if we will only use this result in dimension one, we present it in dimension $n$ for the
sake of completeness (see also \cite{Dan}).

\begin{proposition}\label{dissi}
  For every $\varepsilon \in (0,1)$, for every time
  $t_1, t_2 > 0$ such that $t_2 - t_1 > L$ and for every weak solution
  $\varphi$ of $(S')^n$, the following estimate holds
  \[ \| \varphi (t_1)\|_X \leqslant \exp\left\{- \frac{(t_2 - t_1 -
     L)^2}{4\varepsilon (t_2-t_1)}\right\} \| \varphi (t_2)\|_X . \]
\end{proposition}

\begin{proof}
  We first consider a weight function $\rho (t, x) = \exp(\frac{r}{\varepsilon} x_n)$ for some constant $r \in(0,1)$ which will be fixed later. We will first treat the strong
  solutions case and, using a density argument, we will get the weak solutions
  case.

  We multiply the equation satisfied by $\varphi$ by $\rho \varphi$
  and we integrate on $\Omega$. We get the following identity:
  \[ \frac{1}{2} \frac{d}{dt} \left( \int_\Omega \rho | \varphi
     |^2 \right) =  -\frac{1}{2} \int_\Omega \rho \partial_{x_n}(| \varphi |^2)- \varepsilon \int_{\Omega}
     \rho \varphi \Delta \varphi . \]
  We then integrate by parts in space, which due to $\nabla \rho =  \frac{r}{\varepsilon} \rho e_n$, provides
  \[  -\frac{1}{2} \int_{\Omega} \rho \partial_{x_n} (|\varphi |^2) =  \frac{r}{2\varepsilon} \int_{\Omega} \rho | \varphi|^2 - \frac{1}{2} \left( \int_{\Gamma_0} \rho | \varphi |^2
     - \int_{\Gamma_1} \rho | \varphi |^2 \right), \]
  and
  \[  - \varepsilon \int_{\Omega} \rho \varphi \Delta \varphi =
     \varepsilon \int_{\Omega} \rho |\nabla \varphi|^2 + \frac{ r}{2}
     \int_{\Omega} \rho \partial_{x_n} ( | \varphi |^2) - \varepsilon \left( \int_{\Gamma_0} \rho
     \varphi \partial_{x_n} \varphi  - \int_{\Gamma_1} \rho \varphi \partial_{x_n} \varphi
     \right) \]
  \[ = \varepsilon \int_{\Omega} \rho |\nabla \varphi |^2 - \frac{
     r^2}{2\varepsilon} \int_{\Omega} \rho | \varphi |^2 + \frac{ r}{2} \left( \int_{\Gamma_0} \rho
     |\varphi|^2 - \int_{\Gamma_1} \rho |\varphi|^2
     \right) - \varepsilon \left( \int_{\Gamma_0} \rho
     \varphi \partial_{x_n}  \varphi- \int_{\Gamma_1} \rho \varphi \partial_{x_n} \varphi
     \right) . \]
  Using now the boundary conditions for $\varphi$ and summing up these identities,
  we finally get
  \[ \frac{d}{dt} \left( \int_{\Omega} \rho | \varphi |^2 \right) \geq
     \frac{r(1- r)}{\varepsilon} \int_{\Omega} \rho | \varphi |^2 + (1 -  r)
     \int_{\Gamma} \rho | \varphi |^2 - 2 \varepsilon \int_{\Gamma} \rho \varphi_t \varphi. \]
  On the other hand, it is straightforward that
  \[  \frac{d}{dt} \left( \varepsilon\int_{\Gamma} \rho | \varphi |^2 \right) =
     2 \varepsilon\int_{\Gamma} \rho \varphi_t \varphi, \]
  and, consequently, using that $r \in(0,1)$,
  we have obtained
  \[ \frac{d}{dt} \left( \| \sqrt{\rho (.)} \varphi (t)\|^2_X \right) \geq
     \frac{r(1- r)}{\varepsilon} \| \sqrt{\rho (.)} \varphi
     (t)\|^2_X. \]
  Gronwall's lemma combined with $\exp(-\frac{r}{\varepsilon} L) \leqslant \rho (\cdot) \leqslant 1$ successively gives
  \[ \| \sqrt{\rho (\cdot)} \varphi(\cdot)\|^2_X \leqslant \exp \left( -\frac{r(1- r)}{\varepsilon} (t_2 - t_1)\right) \| \sqrt{\rho (\cdot)} \varphi (t_2)\|^2_X \]
  and
  \[ \| \varphi (t_1)\|^2_X \leqslant \exp \left( - \frac{1}{\varepsilon}\left( r(1- r)(t_2 - t_1)-r L\right)  \right) \| \varphi (t_2)\|^2_X \]
  We finally choose
  $$
  r:=\frac{t_2-t_1-L}{2(t_2-t_1)} \in(0,1),
  $$
  which gives the result.
\end{proof}

We will now use this dissipation estimate with our Carleman inequality to get
the desired result.

\begin{proposition}
  If $n=1$, $\frac{T}{L}$ is sufficiently large and $\varepsilon$ sufficiently small,
  then the observability constant $C_{obs}(\varepsilon)$ is bounded by
  \[ \left. C \exp \left( - \frac{k}{\varepsilon} \right) \right. \]
  where $C$, $k$ \ are some positive constants.
\end{proposition}

\begin{proof}
  \begin{itemize}
    \item We begin with estimating both sides of the Carleman inequality obtained
    above. We use the same notations as above and we define $m = \lambda - e^{2L}$ and $M
    = \lambda-e^{L}$. We first get
    \[ s^9  \int_0^T e^{-4s\alpha(.,0)+2s\alpha(.,-L)}  (\phi^9 |\varphi |^2)(., 0) \lesssim s^7 (\varepsilon T)^{- 14} \exp \left( \frac{s(2M-4m)
       }{(\varepsilon T)^2} \right) \int_0^T | \varphi |^2 (., 0) . \]
    On the other hand, using that $\displaystyle{\phi \gtrsim \frac{1}{(\varepsilon T)^2}}$ on $\left. [
    \frac{T}{4}, \frac{3 T}{4} \right]$, we have the following estimate from below for the left hand-side of the Carleman inequality \eqref{Carleman}
    \[ \frac{s^3}{(\varepsilon T)^6} \exp
       \left( - \frac{2sM}{(\varepsilon T)^2} \right) \left(\int_{\frac{T}{4}}^{\frac{3
       T}{4}} \int_{- L}^0 | \varphi |^2 + \int_{\frac{T}{4}}^{\frac{3 T}{4}}
       \int_{\left. \{- L, 0 \right\}} | \varphi |^2\right). \]
   Consequently we get that
    \[  \| \varphi \|^2_{L^2((T/4,3T/4); X)}
       \lesssim s^4 (\varepsilon T)^{- 8} \exp \left( \frac{4s(M-m)
       }{(\varepsilon T)^2} \right) \int_0^T | \varphi |^2 (., 0):=C \int_0^T | \varphi |^2 (., 0). \]
    We now choose $s \backsim (\varepsilon T)^2 + (\varepsilon T)$. The above constant $C$ is
    consequently estimated by
    \[ \varepsilon^{-4} e^{c / \varepsilon} \lesssim e^{c' / \varepsilon} \]
    for $c' > c$ and $c$ well-chosen.

    \item We now deduce the result using dissipation estimates. We have just proven
    \[ \| \varphi \|^2_{L^2((T/4,3T/4); X)} \lesssim e^{\frac{c'}{\varepsilon}} \int_0^T \int_{\Gamma_0}
       | \varphi |^2 . \]
    We now use the dissipation property with $t_1 = 0$ and $ t_2 = t \in
    \left] \frac{T}{4}, \frac{3 T}{4} \right[$. We easily get, provided $T > 4 L$,
    \[ \frac{T}{2} \exp \left( \frac{(T
       - 4 L)^2}{8 \varepsilon T}  \right) \| \varphi (0) \|^2_X \leqslant \| \varphi \|^2_{L^2((T/4,3T/4); X)}, \]
    which gives the result with $\displaystyle{k = \frac{1}{8} \left( 1-\frac{4L}{T}\right) (T-4L) - c'>0}$ provided that $\displaystyle{\frac{T}{L}>8+32 c'}$.
  \end{itemize}
\end{proof}

\subsection{Proof of Theorem 1}

To show the controllability result for $(S_v)$, we will adopt some minimization
strategy inspired by the classical heat equation.

\begin{proposition}\label{eq}
  A necessary and sufficient condition for the solution of problem $(S_v)$
  to satisfy $u (T) = 0$ is given by:
  \[ \forall \varphi_T \in \text{$X$}, < \varphi (0,.), u_0 >_X = \int_0^T \varphi(.,0) v. \]
  where $\varphi$ is the \ solution of problem $(S')$ with final value
  $\varphi_T$.
\end{proposition}

\begin{proof}
  We apply the Definition \ref{Ladefinition} to $u$ against $\varphi$ a strong
  solution of $(S')$, which gives
  \[  \int_0^T \varphi(.,0) v=< \varphi (0,.), u_0 >_X-<\varphi_T,u(T,.)>_X\]
  and we get the desired equivalence by approximation of weak solutions by
  strong solutions.
\end{proof}

\begin{proposition}\label{ObsControl}
  The following properties are equivalent
  \begin{itemizedot}
    \item $\exists C_1>0,  \forall \varphi_T \in X_{}$; $\|\varphi (0,.)
    \|_X \leq C_1 \| \varphi(.,0) \|_{L^2 (0, T)}$ where
    $\varphi$ is the solution of problem $(S'),$

    \item $\exists C_2>0,  \forall u_0 \in X, \exists v \in L^2(0, T)$ such that $\| v \|_{L^2
    (0, T) } \leqslant C_2 \| u_0 \|_X$ and \
    the solution $u$ of problem $(S_v)$satisfies $u (T) = 0$.
  \end{itemizedot}
  Moreover, $C_1 = C_2$.
\end{proposition}

\begin{proof}
  $(\Rightarrow)$ Let $u_0 \in \text{$X$}$. We define $H$ as the closure of
  $X$ for the norm defined by
  \[ \left. \| \varphi_T \right\|_H = \| \varphi(.,0)\|_{L^2(0,T)}, \]
  where $\varphi$ is the corresponding solution of $(S')$. Using the observability
  assumption and backward uniqueness (Proposition \ref{BU}), one sees that it is
  indeed a norm on $X$.

  We define a functional $J$ in the following way
  \[ J (\varphi_T) = \frac{1}{2} \int_0^T \varphi^2 (.,0) - < \varphi (0,.), u_0 >_X . \]
  $J$ is clearly convex and our assumption imply that $J$ is continuous on $H$. Moreover, thanks to our
  observability assumption, $J$ is coercive. Indeed, one has:
  \[ J (\varphi_T) \geqslant \frac{1}{2} \text{$\| \varphi_T \|^2_H$} - C\|
     \varphi_T \|_H \]
  for $\varphi_T \in H.$
 \par
 \noindent Thus $J$ possesses a global minimum $\widehat{\varphi}_T \in H$, which gives,
  writing Euler-Lagrange equations,
  \begin{equation}\label{Torres}
    \forall \varphi_T \in H, \int_0^T \varphi(.,0) \hat{\varphi}(.,0) =< \varphi (0,.), u_0 >_X .
  \end{equation}
  According to Proposition \ref{eq}, we have shown the existence \ of an
  admissible control defined by $v = \hat{\varphi}(.,0)$.
   Moreover, choosing $\varphi_T = \widehat{\varphi}_T$ in \eqref{Torres}, we obtain the
  following estimate
  \[ \| v \|_{L^2(0,T)} \leqslant \|u_0\|_X \|\hat{\varphi}(0,.)\|_X. \]
  Using our hypothesis, we are done.

  $(\Leftarrow)$ If $v$ is an admissible control with continuous dependence on
  $u_0$, Proposition \ref{eq} gives us, for every $\varphi_T \in X$,
  \[ < \varphi (0,.), u_0 >_X = \int_0^T \varphi(.,0) v. \]
  Choosing now $u_0 = \varphi (0)$ gives us the estimate
  \begin{equation*}
     \| \varphi (0,.) \|^2_X \leqslant C_2 \| \varphi_T
    \|_H  \| u_0 \|_X
  \end{equation*}
  that is
  \[ \| \varphi (0,.) \|_X \leqslant C_2 \| \varphi_T
     \|_H. \]
\end{proof}

\section*{Conclusion and open problems}

The question of controllability in higher dimension is open and seems to be hard to obtain using Carleman estimates.
Another approach to obtain the observability estimate in the $n$-dimensional case could be similar to the one of Miller (see \cite[Theorem 1.5]{M})
but it seems that our problem is not adapted to that framework  (we can check that the link between the one-dimensional problem and the general one is not as simple as it may seem, see Appendix A below).

\par If $n=1$, we can wonder what is the minimal time to get a vanishing control cost
when the viscosity goes to zero. The intuitive result would be $L$, but it seems
that Carleman estimates cannot give such a result.

\appendix
\section{On Miller's trick}
In this paragraph, we detail why we cannot apply Miller's method to deduce observability in dimension $n$ from the one in dimension 1. We refer the reader to the abstract setting of \cite[Section 2]{M} (see more specifically Lemma 2.2 in that reference).
\\We denote by $\Omega^1$ or $\Omega^n$ the domains (resp. $X^1$ or $X^n$ the function spaces,  $\mathcal{A}^1$ or $\mathcal{A}^n $ the evolution operators) in dimension $1$ or $n$. The observability operator in dimension one is defined by
 \[ \mathcal{O}^1:
 \begin{array}{l}
     \varphi \longmapsto \varphi 1_{\{x=0\}}\\
     X^1 \rightarrow L^2 (\partial \Omega^1)
   \end{array}.\]
where $1_{\{x=0\}}$ is the characteristic function of ${\{x=0\}}$.
\\With similar notations, it is obvious that the observability operator on the part of the boundary $\Gamma_0=\mathbb{R}^{n-1} \times \{0\}$ is given by
\[ \mathcal{O}^n:
\begin{array}{l}
     \varphi \longmapsto \varphi 1_{\Gamma_0}\\
     X^n \rightarrow L^2 (\partial \Omega^n)
   \end{array} \]
and coincides with the tensorial product $I\otimes \mathcal{O}^1$ between the identity of $L^2(\mathbb{R}^{n - 1})$ and $\mathcal{O}^1$.

\noindent We have now the following result.
\begin{proposition}
We define the natural ``difference'' between $\mathcal{A}^n $ and $\mathcal{A}^1$ by
\[\mathcal{B} : w \in  L^2 (\mathbb{R}^{n - 1})\mapsto \varepsilon \Delta' w \in L^2 (\mathbb{R}^{n - 1})\]
($\Delta'$ denotes the standard Laplacian on $\mathbb{R}^{n -1}$) with domain
\[\mathcal{D}(\mathcal{B})=H^2 (\mathbb{R}^{n - 1}). \]
Then $\mathcal{A}^n $ does not extend $\mathcal{A}^1\otimes I + I \otimes \mathcal{B}$.
\end{proposition}

\begin{proof}
We compute for  $u_1 \in \mathcal{D} (\mathcal{A}^1)$, $w \in \mathcal{D} (\mathcal{B})$ and $\varphi_1\in X^1$, $\phi \in L^2 (\mathbb{R}^{n - 1})$ the scalar product $<(\mathcal{A}^1 \otimes I+I \otimes \mathcal{B})(u_1 \otimes w),\varphi_1\otimes \phi>_{X^n}$. This value is successively
\[ \begin{array}{l}
<\mathcal{A}^1 u_1,\varphi_1>_{X^1}<w,\phi>_{L^2(\mathbb{R}^{n - 1})}+<u_1,\varphi_1>_{X^1}<Bw, \phi>_{L^2(\mathbb{R}^{n - 1})}\\
\\
=\left(\int_{\Omega^1} ((\varepsilon \partial_{x_n}^2-\partial_{x_n})u_1)\varphi_1  -\int_{\partial\Omega^1} ((\partial_{\nu} +1_{\{x=-L\}}) u_1)\varphi_1\right) \int_{\mathbb{R}^{n - 1}}w \phi \\
\\
+\left(\int_{\Omega^1} u_1 \varphi_1+\varepsilon \int_{\partial\Omega^1} u_1 \varphi_1\right)\varepsilon \int_{\mathbb{R}^{n - 1}} (\Delta'\phi) w\\
\\
=<\mathcal{A}^n(u_1\otimes w),\varphi_1 \otimes \phi>_{X^n}+\varepsilon^2 \int_{\partial\Omega^n} (\Delta'(u_1\otimes w))\varphi_1 \otimes \phi
\end{array}\]
and our claim is proved.
\end{proof}


\begin{thebibliography}{}

\bibitem{C}  \tmtextsc{Coron, J.-M.,} 2007, Control and nonlinearity.
\tmtextit{Mathematical Surveys and Monographs, 136, AMS.}

\bibitem{CG} \tmtextsc{Coron, J.-M., Guerrero, S.,} 2005,
Singular optimal control: a linear 1-D parabolic-hyperbolic example.
\tmtextit{Asymptot. Anal.}, \tmtextbf{44(3-4)}, 237-257.

\bibitem{Dan} \textsc{Danchin, R.,} 1997, Poches de tourbillon visqueuses. {\it J. Math. Pures Appl.}, {\bf 76}, 609-647.

\bibitem{DR} \tmtextsc{Dolecki, S., Russell, D.,} 1977, A general theory
  of observation and control. \tmtextit{SIAM J. Control and Optimization},
  \tmtextbf{15}, No 2, 185-220.

\bibitem{FI} \tmtextsc{Fursikov, A. V., Imanuvilov, O.,} 1996, Controllability of evolution equations. \tmtextbf{Lecture Notes Series, 34}. \tmtextit{Seoul National University, Research Institute of Mathematics, Global Analysis Research Center, Seoul}.

\bibitem{Glass} \tmtextsc{Glass, O.,} 2010, A complex-analytic approach to the problem of uniform controllability of a transport equation in the vanishing viscosity limit.  \tmtextit{J. Funct. Anal.},  \tmtextbf{258}, 852-868.

\bibitem{GG} \tmtextsc{Glass, O., Guerrero, S.,} 2007, On the uniform controllability of the Burgers equation. \tmtextit{SIAM J. Control and Optimization},  \tmtextbf{46}, No 4, 1211-1238.

\bibitem{H} \tmtextsc{Halpern, L.,} 1986, Artificial boundary for the
  linear advection diffusion equation. \tmtextit{Mathematics of Computation},
  \tmtextbf{46}, No 174, 425-438.

\bibitem{HS} \tmtextsc{Halpern, L., Schatzmann, M.,} 1989, Artificial boundary conditions for incompressible viscous flows. \tmtextit{SIAM J. Mathematical Analysis}, \tmtextbf{20},
No. 2, 308--353.

\bibitem{JLL} \textsc{Lions, J.-L.,} 1988, Contr\^olabilit\'e exacte,
stabilisation et perturbations de syst\`emes distribu\'es. \textbf{ RMA 8}, \emph{Masson, Paris.}


\bibitem{LMa} \tmtextsc{Lions, J.-L., Magenes, E.,} 1968, Probl\`emes aux limites
 non homog\`enes et applications, \tmtextit{Vol. I, II}. (French) \tmtextit{Dunod, Paris}.

\bibitem{LM} \tmtextsc{Lions, J.-L., Malgrange, B.,} 1960, Sur l'unicit\'e
  r\'etrograde dans les probl\`emes mixtes paraboliques. (French)
  \tmtextit{Math. Scand.}, \tmtextbf{8}, 277--286.

\bibitem{M} \tmtextsc{Miller, L.,} 2005, On the null-controllability of
  the heat equation in unbounded domains. \tmtextit{Bull Sci Math.}, \tmtextbf{129}, No 2, 175-185.

%\bibitem{T} \tmtextsc{Temam, R.,} 1977, Navier-Stokes equations. Theory and numerical analysis.
%\tmtextit{Studies in Mathematics and its Applications, Vol. 2, North-Holland}.
\end{thebibliography}
\end{document}